\documentclass[11pt]{amsart}
\usepackage{mathrsfs}
\usepackage{color,hyperref}

\input xy
\xyoption{all}

\newcommand{\Fin}{\mathcal{F}\kern-1pt\mathit{in}}
\newcommand{\w}{\omega}

\newcommand{\C}{\mathcal C}
\newcommand{\Ra}{\Rightarrow}
\newcommand{\IZ}{\mathbb Z}
\newcommand{\id}{\mathrm{id}}

\newcommand{\TS}{\mathsf{T\!S}}
\newcommand{\TsL}{\mathsf{T\!sL}}
\newcommand{\TsLw}{\mathsf{T\!sL_w}}
\newcommand{\TsLo}{\mathsf{T\!sL_o}}
\newcommand{\TL}{\mathsf{TL}}

\newcommand{\eC}{\mathsf{e}{:}\mathcal C}
\newcommand{\hC}{\mathsf{h}{:}\mathcal C}

\newcommand{\eTS}{\mathsf{e}{:}\!\mathsf{T\!S}}
\newcommand{\hTS}{\mathsf{h}{:}\!\mathsf{T\!S}}

\newtheorem{theorem}{Theorem}[section]
\newtheorem{mainth}[theorem]{Main Theorem}
\newtheorem{lemma}[theorem]{Lemma}
\newtheorem{corollary}[theorem]{Corollary}
\newtheorem{example}[theorem]{Example}

\newtheorem{proposition}[theorem]{Proposition}

\theoremstyle{definition}
\newtheorem{definition}[theorem]{Definition}

\newtheorem{problem}[theorem]{Problem}

\title[Completeness and absolute $H$-closedness of topological semilattices]{Completeness and absolute $H$-closedness\\ of topological semilattices}
\author{Taras Banakh and Serhii Bardyla}
\address{T.~Banakh: Ivan Franko National University of Lviv (Ukraine) and\newline\indent Jan Kochanowski University in Kielce (Poland)}
\email{t.o.banakh@gmail.com}
\address{S.~Bardyla: Ivan Franko National University of Lviv (Ukraine)}
\email{sbardyla@yahoo.com}

\begin{document}

\begin{abstract} We find (completeness type) conditions on topological semilattices $X,Y$ guaranteeing that each continuous homomorphism $h:X\to Y$ has closed image $h(X)$ in $Y$.
 \end{abstract}
\maketitle
\section{Introduction} It is well-known that a topological group $X$ is complete in its two-sided uniformity if and only if for any isomorphic topological embedding $h:X\to Y$ into a Hausdorff topological group $Y$ the image $h(X)$ is closed in $Y$.

In this paper we prove a similar result for topological semilattices.
A {\em topological semilattice} is a topological space $X$ endowed with a continuous binary operation $X\times X\to X$, $(x,y)\mapsto xy$, which is associative, commutative and {\em idempotent} in the sense that $xx=x$ for all $x\in X$.
\smallskip

We define a Hausdorff topological semigroup $X$ to be
 \begin{itemize}
 \item {\em $H$-closed} if for any isomorphic topological embedding $h:X\to Y$ to a Hausdorff topological semigroup $Y$ the image $h(X)$ is closed in $Y$;
 \item {\em absolutely $H$-closed} if for any continuous homomorphism $h:X\to Y$ to a Hausdorff topological semigroup $Y$ the image $h(X)$ is closed in $Y$.
 \end{itemize}
(Absolutely) $H$-closed topological semigroups were introduced by Stepp in \cite{Stepp1969} (resp. \cite{Stepp1975}) who called them ({\em absolutely}) {\em maximal} topological semigroups. More information on (absolutely) $H$-closed topological semigroups can be found in
\cite{Ban,BanRav2001,Bardyla-Gutik-2012,BardGut-2016(2),
Bardyla-Gutik-Ravsky,Batikova2009,ChuchmanGutik2007,Gutik-2017,Gutik-2014,
Gutik-Pavlyk-2003,GutikPagonRepovs2010,GutikPavlyk2001,GutikPavlyk2003,
GutikRepovs2008,Raikov1946,Ravsky2003,Stepp1969,Stepp1975,Yokoyama2013}.

In this paper we are concentrated at the problem of detecting (absolutely) $H$-closed topological semilattices. For discrete topological semilattices this problem has been resolved by combined efforts of Stepp \cite{Stepp1975} and Banakh, Bardyla \cite{BB17} who proved that a discrete topological semilattice $X$ is absolutely $H$-closed if and only if $X$ is $H$-closed if and only if all chains in $X$ are finite.

A subset $C$ of a semilattice $X$ is called a {\em chain} if $xy\in \{x,y\}$ for all $x,y\in C$. This is equivalent to saying that any two elements of $C$ are comparable in the {\em natural partial order} $\le$ on $X$, defined by $x\le y$ iff $xy=x$. Endowed with this partial order, each semilattice becomes a {\em poset}, i.e., a set  endowed with a partial order. In a Hausdorff topological semilattice $X$ the partial order $\{(x,y)\in X\times X:x\le y\}$ is a closed subset of $X\times X$, which means that $X$ is a {\em pospace}, i.e., a topological space  endowed with a closed partial order. A semilattice $X$ is {\em linear} if the natural partial order on $X$ is linear (i.e., $X$ is a chain in $X$).

The following characterization of (absolutely) $H$-closed discrete topological semilattices is a combined result of Stepp \cite{Stepp1975} and Banakh, Bardyla~\cite{BB17} (who proved the equivalences $(2)\Leftrightarrow(3)$ and $(1)\Leftrightarrow(3)$, respectively).

\begin{theorem}[Stepp, Banakh, Bardyla]\label{t:Stepp} For a discrete topological semilattice $X$ the following conditions are equivalent:
\begin{enumerate}
\item $X$ is $H$-closed;
\item $X$ is absolutely $H$-closed;
\item all chains in $X$ are finite.
\end{enumerate}
\end{theorem}

The implication $(3)\Ra(2)$ in this theorem can be also derived from the following characterization proved by Banakh and Bardyla in \cite{BB17}.

\begin{theorem}[Banakh, Bardyla]\label{t:BB} For a Hausdorff topological semilattice $X$ the following conditions are equivalent:
\begin{enumerate}
\item all closed chains in $X$ are compact;
\item all maximal chains in $X$ are compact;
\item each non-empty chain $C\subset X$ has $\inf C\in\overline{C}$ and $\sup C\in\overline{C}$;
\item each closed subsemilattice of $X$ is absolutely $H$-closed;
\item each closed chain in $X$ is $H$-closed.
\end{enumerate}
\end{theorem}

Here by $\overline{C}$ we denote the closure of a set $C$ is a topological space $X$.

 A subset $A$ of a poset $(X,\le)$ is called
 {\em lower} (resp. {\em upper}) if $A={\downarrow}A$ (resp. $A={\uparrow}A$) where $${\downarrow}A=\{x\in X:\exists a\in A \;(x\le a)\}\mbox{ \ and \ }{\uparrow}A=\{x\in X:\exists a\in A \;(x\ge a)\}.$$

In 2008 Gutik and Repov\v s proved the following characterization of  (absolutely) $H$-closed linear semilattices.

\begin{theorem}[Gutik-Repov\v s]\label{t:GR} For a linear Hausdorff topological semilattice $X$ the following conditions are equivalent:
\begin{enumerate}
\item $X$ is $H$-closed;
\item $X$ is absolutely $H$-closed;
\item any non-empty chain $C\subset X$ has $\inf C\in\overline{{\uparrow}C}$ and $\sup C\in \overline{{\downarrow}C}$ in $X$.
\end{enumerate}
\end{theorem}

In \cite{Yokoyama2013} Yokoyama extended Gutik-Repov\v s Theorem~\ref{t:GR} to topological pospaces with finite antichains.

Trying to extend Gutik-Repov\v s characterization to all (not necessarily linear) topological semilattices we have discovered that the last condition of Theorem~\ref{t:GR} admits at least three non-equivalent versions, defined as follows.

\begin{definition}\label{d:complete} A topological semilattice $X$ is defined to be
\begin{itemize}
\item {\em $k$-complete} if each non-empty chain $C\subset X$ has $\inf C\in \overline C$ and $\sup C\in\overline C$;
\item {\em $s$-complete} if each non-empty subsemilattice $S\subset X$ has $\inf S\in\overline{{\uparrow}S}$ and each non-empty chain $C\subset X$ has $\sup C\in\overline{{\downarrow}C}$;
\item {\em $c$-complete} if for each closed upper set $F\subset X$, each non-empty chain $C\subset F$ has $\inf C\in F$ and $\sup C\in \overline{{\downarrow}C}$ in $X$.
\end{itemize}
\end{definition}
In Proposition~\ref{p:diag} we shall prove that for any topological semilattice the following implications hold:
$$\mbox{$k$-complete $\Ra$ $s$-complete $\Ra$ $c$-complete}.$$
By Theorem~\ref{t:BB}, a Hausdorff topological semilattice is $k$-complete if and only if all maximal chains in $X$ are compact. Theorems~\ref{t:Stepp} and \ref{t:GR} imply that a discrete or linear Hausdorff topological semilattice $X$ is $s$-complete if and only if $X$ is $c$-complete if and only if $X$ is (absolutely) $H$-closed.

These completeness properties of topological semilattices will be paired with the following notions.

\begin{definition} We say that a pospace $(X,\le)$ is
\begin{itemize}
\item {\em well-separated} if for any points $x<y$ there exists a neighborhood $V\subset X$ of $y$ such that $x\notin{\uparrow}\overline{{\uparrow}\overline{{\uparrow}V}}$;
\item {\em down-open} if for every open set $U\subset X$ the lower set ${\downarrow}U$ is open in $X$;
\item a {\em topological lattice} if any two points $x,y\in X$ have $\inf\{x,y\}$ and $\sup\{x,y\}$ and the binary operations $\wedge:X\times X\to X$, $\wedge:(x,y)\mapsto\inf\{x,y\}$, and $\vee:X\times X\to X$, $\vee:(x,y)\mapsto\sup\{x,y\}$ are continuous.
\end{itemize}
\end{definition}

In Lemma~\ref{l:well} we shall prove that for a Hausdorff topological semilattice the following implications hold:
$$\mbox{linear $\Ra$ topological lattice $\Ra$ down-open $\Ra$ well-separated.}$$

The main result of this paper is the following theorem, which will be proved in Section~\ref{s:proof}, after some preparatory work made in Sections~\ref{s:ul}--\ref{s:tech}.

\begin{mainth}\label{main} Let $h:X\to Y$ be a continuous homomorphism from a topological semilattice $X$ to a Hausdorff topological semigroup $Y$. The image $h(X)$ is closed in $Y$ if one of the following conditions is satisfied:
\begin{enumerate}
\item $X$ is $k$-complete;
\item $X$ is $s$-complete and $Y$ is well-separated;
\item $X$ is $c$-complete and $X$ or $Y$ is down-open.
\end{enumerate}
\end{mainth}



Some corollaries of Theorem~\ref{main} can be formulated using the notion of a $\vec\C$-closed topological semilattice for a category $\vec\C$ whose objects are topological semigroups and morphisms are continuous homomorphisms between topological semigroups.

Given a class $\C$ of topological semigroups by $\hC$ and $\eC$ we denote the category whose objects are topological semigroups in the class $\C$ and morphisms are continuous homomorphisms and isomorphic topological embeddings of topological semigroups in the class $\C$, respectively.

\begin{definition} Let $\vec\C$ be a category whose objects are topological semigroups and morphisms are continuous homomorphisms between topological semigroups.
An object $X$ of the category $\vec\C$ is called {\em $\vec\C$-closed} if for any morphism $h:X\to Y$ of the category $\vec\C$ the image $h(X)$ is closed in $Y$.

In particular, for a class $\C$ of topological semigroup, a topological semigroup $X\in\C$ is called
\begin{itemize}
\item {\em $\hC$-closed} if for any continuous homomorphism $f:X\to Y\in\C$ the image $f(X)$ is closed in $Y$;
\item {\em $\eC$-closed} if for each isomorphic topological embedding $f:X\to Y\in\C$ the image $f(X)$ is closed in $Y$.
\end{itemize}
\end{definition}
Therefore a topological semigroup $X$ is (absolutely) $H$-closed if and only if it is $\eTS$-closed (resp. $\hTS$-closed) for the class $\TS$ of Hausdorff topological semigroups.

We shall be interested in $\hC$-closedness for the following classes $\C$ of topological semilattices:
\begin{itemize}
\item $\TsL$ of all Hausdorff topological semilattices;
\item $\TsLw$ of all well-separated Hausdorff topological semilattices;
\item $\TsLo$ of all down-open Hausdorff topological semilattices;
\item $\TL$ of Hausdorff topological lattices.
\end{itemize}

The results of this paper allow us to draw the following diagram, containing implications between various completeness and closedness properties of a Hausdorff topological semilattice. The implications from this diagram will be proved in Section~\ref{s:diagram}.
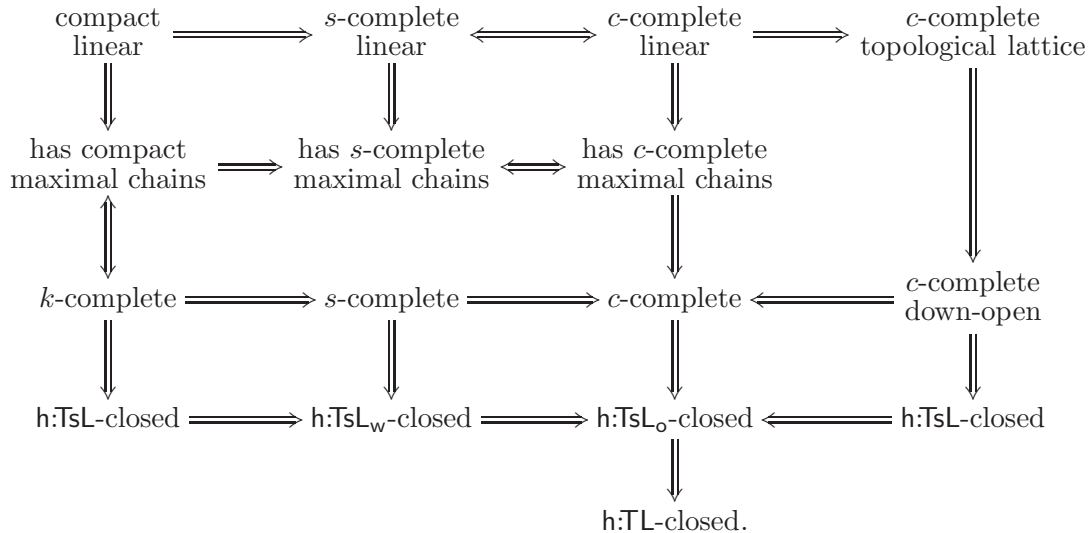
$$\xymatrix{
\mbox{compact}\atop\mbox{linear}\ar@{=>}[r]\ar@{=>}[d]&
\mbox{$s$-complete}\atop\mbox{linear}\ar@{<=>}[r]\ar@{=>}[d]&
\mbox{$c$-complete}\atop\mbox{linear}\ar@{=>}[r]\ar@{=>}[d]&
\mbox{$c$-complete}\atop\mbox{topological lattice}\ar@{=>}[dd]\\
\mbox{has compact}\atop\mbox{maximal chains}\ar@{=>}[r]\ar@{<=>}[d]&
\mbox{has $s$-complete}\atop\mbox{maximal chains}\ar@{<=>}[r]&
\mbox{has $c$-complete}\atop\mbox{maximal chains}\ar@{=>}[d]\\
\mbox{$k$-complete}\ar@{=>}[r]\ar@{=>}[d]&
\mbox{$s$-complete}\ar@{=>}[r]\ar@{=>}[d]&\mbox{$c$-complete}\ar@{=>}[d]
&\mbox{$c$-complete}\atop\mbox{down-open}\ar@{=>}[l]\ar@{=>}[d]\\
\mbox{${\mathsf h}{:}\!\TsL$-closed}\ar@{=>}[r]&
\mbox{${\mathsf h}{:}\!\TsLw$-closed}\ar@{=>}[r]&
\mbox{${\mathsf h}{:}\!\TsLo$-closed}\ar@{=>}[d]&\mbox{${\mathsf h}{:}\!\TsL$-closed}\ar@{=>}[l]\\
&&
\mbox{${\mathsf h}{:}\!\TL$-closed}.
}
$$

\section{Upper and lower sets in topological semilattices}\label{s:ul}

In this section we prove some auxiliary results related to upper and lower sets in topological semilattices.

Simple examples show that in general, the closure of an (upper) lower set in a pospace is not necessarily an (upper) lower set in the pospace. However, in semilattices we have the following positive result.

\begin{lemma}\label{l:updown} Let $X$ be a topological semilattice.
\begin{enumerate}
\item For any open set $U\subset X$ the upper set ${\uparrow}U$ is open in $X$.
\item The interior of any upper set $U\subset X$ is an upper set in $X$.
\item The closure $\bar L$ of any lower set $L\subset X$ is a lower set in $X$.
\end{enumerate}
\end{lemma}

\begin{proof} Given an open set $U\subset X$, for every point $x\in{\uparrow}U$ choose a point $u\in U$ with $u\le x$ and observe that $O_x=\{z\in X:uz\in U\}$ is an open neighborhood of $x$, contained in ${\uparrow}U$.

Given any upper set $P\subset X$ consider its interior $P^\circ$ in $X$ and observe that for any point $x\in P^\circ$ there exists an open set $U\subset P$ containing the point $x$. Taking into account that the upper set ${\uparrow}U$ is open, we conclude that ${\uparrow}x\subset{\uparrow}U\subset P^\circ$, which means that $P^\circ$ is an upper set in $X$.

For any lower set $L\subset X$, the complement $X\setminus L$ is an upper set, whose interior $(X\setminus L)^\circ$ is an upper set in $X$. Then the complement $X\setminus (X\setminus L)^\circ=\bar L$ is a lower set in $X$.
\end{proof}

Applying Lemma~\ref{l:updown}(1) to the continuous operation $(x,y)\mapsto\sup\{x,y\}$ in a topological lattice, we obtain the following simple (but important) fact.

\begin{corollary}\label{c:l=>do} Each topological lattice $X$ is down-open.
\end{corollary}
The following proposition is a straightforward corollary of \cite[Lemma~1]{Ward1954}.

\begin{proposition}\label{p:Ward} Each linear pospace is a topological lattice.
\end{proposition}

\begin{lemma}\label{l:neib}  Any points $x\not\le y$ of a pospace $X$ have open neighborhoods $V_x,V_y\subset X$ such that ${\uparrow}V_x\cap{\downarrow}V_y=\emptyset$.
\end{lemma}

\begin{proof} Since the partial order $\le$ is closed, the points $x,y$ have open neighborhoods $V_x,V_y$ such that $V_x\times V_y$ is disjoint with the partial order $\le$ in $X\times X$. Consequently, $\tilde x\not\le\tilde y$ for any $\tilde x\in V_x$ and $\tilde y\in V_y$. We claim that ${\uparrow}V_x\cap {\downarrow}V_y=\emptyset$. Assuming that this intersection contains some point $v$, we could find points $\tilde x\in V_x$ and $\tilde y\in V_y$ such that $\tilde x\le v\le \tilde y$, which contradicts the choice of $V_x$ and $V_y$.
\end{proof}

\begin{corollary} Any points $x\not\le y$ in Hausdorff topological semilattice $X$ have open neighborhoods
$V_x,V_y\subset X$ such that ${\uparrow}V_x\cap{\downarrow}V_y=\emptyset$.
\end{corollary}

\section{Chain-closed sets in semilattices}\label{s2}

A subset $A$ of a poset $X$ is called
\begin{itemize}
\item {\em lower chain-closed in $X$} if $\inf C\in A$ for any chain $C\subset A$ possessing $\inf C\in X$;
\item {\em upper chain-closed in $X$} if $\sup C\in A$ for any chain $C\subset A$ possessing $\sup C\in X$;
\item {\em chain-closed in $X$} if $A$ is lower chain-closed and upper chain-closed in $X$.
\end{itemize}

A poset $(X,\le)$ is defined to be {\em chain-complete} if each chain $C\subset X$ has $\inf C$ and $\sup C$ in $X$. Definition~\ref{d:complete} implies that each $c$-complete topological semilattice is chain-complete.

The following lemma can be easily derived from the definitions.

\begin{lemma}\label{c:fid=>cc} Any closed upper or lower set in a $c$-complete topological semilattice is chain-closed.
\end{lemma}


For two subsets $A,B$ of a semilattice $X$ let $AB=\{ab:a\in A,\;b\in B\}$ be their product in $X$.

The proof of the following simple fact can be found in
 Lemma III-1.2 of \cite{GHKLMS}.

\begin{lemma}\label{l:inf-c} Let $X$ be a semilattice and $A,B\subset X$ be two sets that have $\inf A$ and $\inf B$ in $X$. Then $(\inf A)\cdot(\inf B)=\inf(A\cdot B)$.
\end{lemma}

A  semilattice $X$ is called {\em chain-continuous} if for any chain $C\subset X$ possessing $\sup C\in X$ we get $x\cdot\sup C=\sup xC$.

\begin{lemma}\label{l:sup-clo} A Hausdorff topological semilattice $X$ is chain-continuous if $\sup C\in\overline{{\downarrow}C}$ for any chain $C\subset X$ possessing $\sup C\in X$. Consequently, each $c$-complete Hausdorff topological semilattice is chain-continuous.
\end{lemma}

\begin{proof} Let $C\subset X$ be a chain possessing $\sup C\in X$. It is easy to see that for any $a\in X$ the product $a\cdot\sup C$ is an upper bound for the set $aC$. To show that $a\cdot\sup C=\sup aC$, we need to check that $a\cdot\sup C\le b$ for any upper bound $b\in X$ of the set $aC$ in $X$. Taking into account that the set ${\downarrow}b$ is a closed lower set in $X$, we conclude that $\sup aC\in \overline{{\downarrow}(aC)}\subset{\downarrow}b$ and hence $a\cdot\sup C\le b$ and $a\cdot\sup C=\sup aC$.
\end{proof}


For a subset $A$ of a poset $X$ by the {\em chain-closure} $\overleftrightarrow{A}$ of $A$ we understand the smallest chain-closed subset of $X$ that contain the set $A$. It is equal to the intersection of all chain-closed subsets of $X$ containing $A$.

\begin{lemma}\label{l:sup} Let $X$ be a chain-continuous semilattice. For any point $a\in X$ and set $B\subset X$ we get $$a\cdot\overleftrightarrow{B}\subset \overleftrightarrow{aB}.$$
\end{lemma}

\begin{proof} Consider the shift $s_a:X\to X$, $s_a:x\mapsto ax$. We claim that the set $P=s_a^{-1}(\overleftrightarrow{aB})$ is chain-closed in $X$. Indeed, for any chain $C\subset P$ possessing $\inf C$,  Lemma~\ref{l:inf-c} implies that $a\cdot\inf C=\inf aC$. Since the set  $\overleftrightarrow{aB}\supset s_a(P)\supset s_a(C)=aC$ is chain-closed in $X$, $s_a(\inf C)=a\cdot\inf C=\inf aC\in
\overleftrightarrow{aB}$ and hence $\inf C\in P$.

By analogy we can prove that for any chain $C\subset P$ possessing $\sup C$ in $X$, we get $\sup C\in P$. This means that the set $P$ is chain-closed, contains $B$ and hence $\overleftrightarrow{B}\subset P$, which implies $a\overleftrightarrow{B}=s_a(\overleftrightarrow{B})\subset s_a(P)\subset\overleftrightarrow{aB}$.
\end{proof}

\begin{corollary}\label{c:sup} For any subsets $A,B\subset X$ of a chain-continuous semilattice $X$, we get $$\overleftrightarrow{A}\cdot\overleftrightarrow{B}\subset \overleftrightarrow{AB}.$$
\end{corollary}

\begin{proof} For any point $a\in A$, by Lemma~\ref{l:sup}, we get $a\overleftrightarrow{B}\subset\overleftrightarrow{aB}\subset\overleftrightarrow{AB}$ and hence $A\overleftrightarrow{B}\subset\overleftrightarrow{AB}$.
Applying Lemma~\ref{l:sup} once more, we conclude that for every $b\in\overleftrightarrow{B}$, we get
$$\overleftrightarrow{A}b\subset \overleftrightarrow{Ab}\subset \overleftrightarrow{A\overleftrightarrow{B}}\subset
\overleftrightarrow{\overleftrightarrow{AB}}=\overleftrightarrow{AB}$$and hence $\overleftrightarrow{A}\cdot\overleftrightarrow{B}\subset\overleftrightarrow{AB}$.
\end{proof}

\section{Countable chain-closures of sets in semilattices}

For a subset $A$ of a chain-complete poset $X$ let $$
\begin{aligned}
&\bar\downarrow A=\{\inf C:\mbox{$C$ is a countable chain in $A$}\},\\
&\bar\uparrow A=\{\sup C:\mbox{$C$ is a countable chain in $A$}\}, \mbox{ and}\\
&\bar{\uparrow}\bar{\downarrow}A=\bar{\uparrow}({\bar\downarrow}A).
\end{aligned}
$$

\begin{lemma}\label{l:cccl} For any subsets $A,B\subset G$ of a semilattice $X$ we get $\bar\downarrow A\cdot\bar\downarrow B\subset\bar\downarrow(AB)$. If the semilattice $X$ is chain-continuous, then $\bar\uparrow A\cdot\bar\uparrow B\subset\bar\uparrow(AB)$.
\end{lemma}

\begin{proof} Given two points $a\in\bar\downarrow A$ and $b\in\bar\downarrow B$, find countable chains $\vec A\subset A$ and $\vec B\subset B$ such that $a=\inf \vec A$ and $b=\inf \vec B$. Since $\vec A$ and $\vec B$ are countable, we can find decreasing sequences $\{a_n\}_{n\in\w}\subset \vec A$ and $\{b_n\}_{n\in\w}\subset \vec B$ such that $a=\inf\{a_n\}_{n\in\w}$ and $b=\inf\{b_n\}_{n\in\w}$. It follows that $\{a_nb_n\}_{n\in\w}$ is a decreasing sequence in $AB$. Moreover, for any $m,k\in\w$ there exists $n\in\w$ such that $a_nb_n\le a_mb_k$, which implies that $ab\le \inf\{a_nb_n\}_{n\in\w}=\inf\{a_nb_k\}_{n,k\in\w}=\inf (\vec A \vec B)$.

 By Lemma~\ref{l:inf-c}, for every $n\in\w$, $$a_nb=a_n\cdot\inf\{b_k\}_{k\in\w}=\inf\{a_nb_k\}_{k\in\w}\ge \inf\{a_mb_k\}_{m,k\in\w}=\inf(\vec A\vec B).$$ Applying the Lemma~\ref{l:inf-c} once more, we get
 $ab=\inf_{n\in\w}a_nb\ge \inf (\vec A\vec B)$ and hence $$ab=\inf (\vec A\vec B)=\inf\{a_nb_n\}_{n\in\w}\in \bar\downarrow{AB}.$$
 By analogy we can prove that the chain-continuity of $X$ implies $\bar\uparrow A\cdot\bar\uparrow B\subset\bar\uparrow(AB)$.
 \end{proof}

\begin{lemma}\label{l:w} Let $\{V_n\}_{n\in\w}$ be a decreasing  sequence of non-empty sets in a chain-continuous chain-complete semilattice $X$ such that $V_{n+1}V_{n+1}\subset V_n$ for all $n\in\w$. Then the set $L=\bigcap_{n\in\w}{\bar{\uparrow}\bar\downarrow V_n}$ is a non-empty subsemilattice of $X$.
\end{lemma}

\begin{proof}
To see that the set $L$ is a subsemilattice of $X$, observe that for every positive integer $n$, Lemma~\ref{l:cccl} guarantees that $$LL\subset (\bar\uparrow\bar\downarrow V_n)\cdot(\bar\uparrow\bar\downarrow V_n)
\subset \bar\uparrow(\bar\downarrow V_n\cdot\bar\downarrow V_n)\subset\bar\uparrow\bar\downarrow(V_nV_n)\subset\bar\uparrow\bar\downarrow V_{n-1}$$ and hence, $LL\subset\bigcap_{n\in\w}\bar\uparrow\bar\downarrow V_{n}=L$, which means that $L$ is a subsemilattice of $X$.
\smallskip

To see that $L$ is not empty, for every $n\in\w$ choose a point $v_n\in V_{n+1}$. For any numbers $n<m$ consider the product $v_n\cdots v_m\in X$. Using the inclusions $V_kV_k\subset V_{k-1}$ for $k\in\{m,m-1,\dots,n+1\}$, we can show that $v_n\cdots v_m\in V_n$. Observe that for every $n\in\w$ the sequence $(v_{n}\cdots v_m)_{m>n}$ is decreasing and by the chain-completeness of $X$, this chain has the greatest lower bound $\bar v_n:=\inf_{m>n}v_n\cdots v_m$ in $X$, which belongs to $\bar\downarrow{V_n}$.

Taking into account that $v_n\cdots v_m\le v_k\cdots v_m$ for any $n\le k<m$, we conclude that $\bar v_n\le \bar v_k$ for any $n\le k$. By the chain-completeness of $X$ the chain $\{\bar v_n\}_{n\in\w}$ has the least upper bound $\bar v=\sup\{\bar v_n\}_{n\in\w}$ in $X$. Since for every $k\in\w$ the chain $\{\bar v_n\}_{n>k}$ is contained in $\bar\downarrow{V_k}$, the least upper bound $\bar v=\sup\{\bar v_n\}_{n\in\w}=\sup\{\bar v_n\}_{n\ge k}$ belongs to $\bar\uparrow\bar\downarrow{V_k}$ for all $k\in\w$. Then $\bar v\in\bigcap_{n\in\w}\bar\uparrow\bar\downarrow{V_n}=L$, so $L\ne\emptyset$.
\end{proof}

\section{The Key Lemma}

In this section we shall prove a Key Lemma for the proof of Theorem~\ref{main}.

\begin{lemma}\label{l:super} Let $h:X\to Y$ be a homomorphism of topological semilattices. The subsemilattice $S=h(X)\subset Y$ is closed in $Y$ if the following conditions are satisfied:
\begin{enumerate}
\item $X$ is chain-complete and chain-continuous;
\item for any $y\in \bar S\setminus S\subset Y$ and $x\in X$ there exists a sequence $(U_n)_{n\in\w}$ of neighborhoods of $y$ in $Y$ such that such that the point $x$ does not belong to the chain-closure $\overleftrightarrow{U}$ of the set $U=\bigcap_{n\in\w}\bar{\uparrow}\bar{\downarrow}h^{-1}(U_n)$ in $X$.
\end{enumerate}
\end{lemma}

\begin{proof} Given any point $y\in \bar S$, we should prove that $y\in S$. To derive a contradiction, assume that $y\notin S$. By transfinite induction, for every non-zero cardinal $\kappa$ we shall prove the following statement:
\begin{itemize}
\item[$(*_\kappa)$] {\em for every family $\{U_{\alpha,n}\}_{(\alpha,n)\in\kappa\times \w}$ of neighborhoods of $y$ in $Y$, the set $\bigcap_{\alpha\in\kappa}\overleftrightarrow{\bigcap_{n\in\w}\bar\uparrow\bar\downarrow h^{-1}(U_{\alpha,n})}$ is not empty.}
\end{itemize}
\smallskip

To prove this statement for the smallest infinite cardinal $\kappa=\w$, fix an arbitrary double sequence $(U_{\alpha,n})_{(\alpha,n)\in\w\times\w}$ of neighborhoods of $y$. Using the continuity of the semilattice operation at $y$, construct a decreasing sequence $(W_n)_{n\in\w}$ of neighborhoods of $y$ such that $W_n\subset \bigcap_{\alpha+k\le n}U_{\alpha,k}$ and $W_{n+1}W_{n+1}\subset W_n$ for all $n\in\w$. Then the preimages $V_n=h^{-1}(W_n)$, $n\in\w$, form a decreasing sequence $(V_n)_{n\in\w}$ of non-empty sets in $X$ such that $V_nV_n\subset V_{n-1}$ for all $n>0$. By Lemma~\ref{l:w}, the set $\bigcap_{n\in\w}\bar\uparrow\bar\downarrow V_n$ is not empty.
Then the set
$$
\bigcap_{\alpha\in\w}\overleftrightarrow
{\bigcap_{n\in\w}\bar\uparrow\bar\downarrow h^{-1}(U_{\alpha,n})}\supset\bigcap_{\alpha\in\w}\bigcap_{n\in\w}
\bar\uparrow\bar\downarrow h^{-1}(U_{\alpha,n})\supset\bigcap_{n\in\w}\bar\uparrow\bar\downarrow h^{-1}(W_n)=\bigcap_{n\in\w}\bar\uparrow\bar\downarrow V_n$$
is not empty as well.

Now assume that for some infinite cardinal $\kappa$ and all cardinals $\lambda<\kappa$ the statement $(*_\lambda)$ has been proved. To prove the statement $(*_\kappa)$, fix any family $(U_{\alpha,n})_{(\alpha,n)\in\kappa\times\w}$ of neighborhoods of $y$ in the topological semilattice $Y$.

Using the continuity of the semilattice operation on $Y$,
 for every $\alpha\in\kappa$ choose a decreasing sequence $(W_{\alpha,n})_{n\in\w}$ of neighborhoods of $y$ such that
$W_{\alpha,n}\subset U_{\alpha,n}$ and $W_{\alpha,n+1}^2\subset W_{\alpha,n}$ for all $n\in\w$. By Lemma~\ref{l:w}, the intersection  $L_\alpha=\bigcap_{n\in\w}\bar\uparrow\bar\downarrow{h^{-1}(W_{\alpha,n})}
\subset\bigcap_{n\in\w}\bar\uparrow\bar\downarrow{h^{-1}(U_{\alpha,n})}$ is a non-empty  subsemilattice in $X$. Lemma~\ref{c:sup} implies that
$\overleftrightarrow{L_\alpha}\cdot\overleftrightarrow{L_\alpha}\subset\overleftrightarrow{L_\alpha L_\alpha}=\overleftrightarrow{L_\alpha}$, which means that $\overleftrightarrow{L_\alpha}$ is a chain-closed subsemilattice in $X$.
Then for every $\beta\in\kappa$ the intersection $$L_{<\beta}=\bigcap_{\alpha<\beta}\overleftrightarrow{L_\alpha}=
\bigcap_{\alpha<\beta}\overleftrightarrow{\bigcap_{n\in\w}
\bar\uparrow\bar\downarrow h^{-1}(W_{\alpha,n})}$$ is a chain-closed subsemilattice of $X$. By the inductive assumption $(*_{|\beta\times\w|})$, the semilattice $L_{<\beta}$ is not empty.

Choose any maximal chain $M$ in $L_{<\beta}$. The chain-completeness of $X$ guarantees that $M$ has $\inf M\in X$. Taking into account that $L_{<\alpha}$ is chain-closed in $X$, we conclude that $\inf M\in L_{<\alpha}$. We claim that $x_\alpha:=\inf M$ is the smallest element of the semilattice $L_{<\alpha}$. In the opposite case, we could find an element $z\in L_{<\alpha}$ such that $x_\alpha\not\le z$ and hence $x_\alpha z<x_\alpha$. Then $\{x_\alpha z\}\cup M$ is a chain in $L_{<\alpha}$ that properly contains the maximal chain $M$, which is not possible. This contradiction shows that $x_\alpha=\inf M$ is the smallest element of the semilattice $L_{<\alpha}$.

Observe that for any ordinals $\alpha<\beta<\kappa$ the inclusion $L_{<\beta}\subset L_{<\alpha}$ implies $x_\alpha\le x_\beta$. So, $\{x_\beta\}_{\beta\in\kappa}$ is a chain in $X$ and by the chain-completeness of $X$, it has $\sup\{x_\beta\}_{\beta\in\kappa}\in \bigcap_{\beta<\kappa}L_{<\beta}\subset\bigcap_{\alpha\in\kappa}
\overleftrightarrow{\bigcap_{n\in\w}\bar\uparrow\bar\downarrow h^{-1}(U_{\alpha,n})}$. So, the latter set is not empty and the statement $(*_\kappa)$ is proved.
\smallskip

By condition (2), for every point $x\in X$ there exists a sequence $(U_{x,n})_{n\in\w}$ of open neighborhoods of $y$ in $Y$ such  that $x\notin\overleftrightarrow{\bigcap_{n\in\w}\bar\uparrow\bar\downarrow h^{-1}(U_{x,n})}$. Then the set $\bigcap_{x\in X}\overleftrightarrow{\bigcap_{n\in\w}\bar\uparrow\bar\downarrow h^{-1}(U_{x,n})}$ is empty, which contradicts the property $(*_\kappa)$ for $\kappa=|X|$.
\end{proof}

\section{Well-separated semilattices}\label{s:well}

We recall that a topological semilattice $X$ is {\em well-separated} if for any points $x<y$ in $X$ there exists a neighborhood $U\subset X$ of $y$ such that $x\notin{\uparrow}\overline{{\uparrow}\overline{{\uparrow} U}}$.

\begin{lemma}\label{l:well} A Hausdorff topological semilattice $X$ is well-separated if one of the following conditions is satisfied:
\begin{enumerate}
\item $X$ admits a continuous injective homomorphism into a well-separated topological semilattice;
\item the pospace $X$ is down-open;
\item $X$ is a topological lattice.
\end{enumerate}
\end{lemma}

\begin{proof}
1. Assume that $h:X\to Y$ is a continuous injective homomorphism of $X$ into a well-separated topological semilattice $Y$. Given any points $x,y\in X$ with $x<y$, consider their images $\bar x=h(x)$ and $\bar y= h(y)$ and observe that $\bar x<\bar y$ (by the injectivity of $h$). Since $Y$ is well-separated, the point $\bar y$ has a neighborhood $V\subset Y$ such that $\bar x\notin{\uparrow}\overline{{\uparrow}\overline{{\uparrow}V}}$. By the continuity and monotonicity of $h$, the neighborhood $U:=h^{-1}(V)$ of $y$ witnesses that $X$ is well-separated as $x\notin {\uparrow}\overline{{\uparrow}\overline{{\uparrow}U}}$.
\smallskip

2. Assume that the lower set of any open set in $X$ is open. To show that $X$ is well-separated, take any points $x,y\in X$ with $x<y$. By Lemma~\ref{l:neib}, the points $x,y$ have open neighborhoods $V_x,V_y$ in $X$ such that ${\downarrow}V_x\cap{\uparrow}V_y=\emptyset$. By our assumption, the lower set ${\downarrow}V_x$ is open in $X$ and  its complement $F=X\setminus{\downarrow}V_x$ is a closed upper set. Then for the neighborhood $V_y\subset F$ of $y$ we get $${\uparrow}\overline{{\uparrow}\overline{{\uparrow}V_y}}\subset
{\uparrow}\overline{{\uparrow}\overline{{\uparrow}F}}=F\subset X\setminus\{x\}.$$
\smallskip

3. If $X$ is a topological lattice, then by Corollary~\ref{c:l=>do}, $X$ is down-open and by the preceding item, the topological semilattice $X$ is well-separated.
\end{proof}

\begin{problem} Find an example of a Hausdorff topological semilattice $X$ which is not well-separated.
\end{problem}

\begin{problem} Is each regular topological semilattice well-separated?
\end{problem}

\section{A Main Technical Result}\label{s:tech}


\begin{theorem}\label{t:main}
For a continuous homomorphism $h:X\to Y$ from a $c$-complete topological semilattice $X$ to a Hausdorff topological semigroup $Y$, the image $h(X)$ is closed in $Y$ if one of the following conditions is satisfied:
\begin{enumerate}
\item $X$ or $Y$ is a down-open topological semilattice;
\item $\overleftrightarrow{U}\subset\overline{{\uparrow}U}$ for any open set $U\subset X$;
\item $Y$ is well-separated topological semilattice and $\overleftrightarrow{S}\subset{\uparrow}\overline{{\uparrow}S}$ for any subsemilattice $S\subset X$.
\end{enumerate}
\end{theorem}

\begin{proof} Observing that the closure $\bar S$ of the semilattice $S$ in the Hausdorff topological semigroup $Y$ is a semilattice, we can assume that $Y=\bar S$ is a topological semilattice.

Being $c$-complete, the semilattice $X$ is chain-complete. By Lemma~\ref{l:sup-clo}, the topological semilattice $X$ is  chain-continuous.
The closedness of the set $S:=h(X)$ in $Y$ will follow from Lemma~\ref{l:super} as soon as we check the condition (2) of this lemma.

Given any points $y\in\bar S\setminus S\subset Y$ and $x\in X$ we need to find a sequence $(U_n)_{n\in\w}$ of neighborhoods of $y$ in $Y$ such that $x\notin \overleftrightarrow{\bigcap_{n\in\w}\bar\uparrow\bar\downarrow h^{-1}(U_n)}$. Depending on the relation between the points $y$ and $s:=h(x)$, we consider two cases.
\smallskip

If $s\not\le y$, then by Lemma~\ref{l:neib}, the points $s$ and $y$ have neighborhoods $V_s$ and $V_y$ in $Y$ such that ${\uparrow}V_s\cap{\downarrow}V_y=\emptyset$. By Lemma~\ref{l:updown}(1), the upper set ${\uparrow}V_s$ is open in $Y$.
Then $F:=h^{-1}(Y\setminus{\uparrow}V_s)$ is a closed lower set in $X$ that contains $h^{-1}(V_y)$ but does not contain the point $x\in h^{-1}(V_s)$. By Lemma~\ref{c:fid=>cc}, the closed lower set $F$ is chain-closed in $X$. Then for the neighborhoods $U_n:=V_y$, $n\in\w$ we get $$\overleftrightarrow{\bigcap_{n\in\w}\bar\uparrow\bar\downarrow h^{-1}(U_{n})}=\overleftrightarrow{\bar\uparrow\bar\downarrow h^{-1}(V_y)}=\overleftrightarrow{h^{-1}(V_y)}\subset \overleftrightarrow{F}=F\subset X\setminus \{x\}.$$
\smallskip

The case $s\le y$ is more complicated. In this case $y\not\le s$ and we can apply Lemma~\ref{l:neib} to find two open neighborhoods $V_s,V_y\subset Y$ of $s,y$ such that  ${\downarrow}V_s\cap{\uparrow} V_y=\emptyset$. By Lemma~\ref{l:updown}(1), the upper set ${\uparrow}V_y$ is open in $Y$. The continuity of the homomorphism $h$ implies that $h^{-1}({\uparrow}V_y)$ is an open upper set in $X$.

If $X$ is down-open, then the lower set ${\downarrow}h^{-1}(V_s)$ is open in $X$ and $X\setminus{\downarrow}h^{-1}(V_s)$ is an upper closed set that contains $h^{-1}(V_y)$ and is disjoint with $h^{-1}(V_s)$.

If $Y$ is down-open, then the lower set ${\downarrow}V_s$ is open in $Y$ and the preimage $h^{-1}({\downarrow}V_s)$ is an open lower set in $X$. Then its complement $X\setminus h^{-1}({\downarrow}V_s)$ is an upper closed set that contains $h^{-1}(V_y)$ and is disjoint with $h^{-1}(V_s)$.

In both cases we have found an upper closed set $F\subset X$, containing the set $h^{-1}(V_y)$ and disjoint with the set $h^{-1}(V_s)\ni x$. By the $c$-completeness of $X$, the upper closed set $F$ is chain-closed. For every $n\in\w$ put $U_n:=V_y$ and observe that
 $$\overleftrightarrow{\bigcap_{n\in\w}\bar\uparrow\bar\downarrow h^{-1}(U_n)}\subset \overleftrightarrow{\bar\uparrow\bar\downarrow h^{-1}(V_y)}= \overleftrightarrow{h^{-1}(V_y)}\subset  \overleftrightarrow{F}=F\subset X\setminus\{x\}.$$
\smallskip

If the condition (2) of the theorem holds, then we put $U_n:={\uparrow}V_y$ for all $n\in\w$ and observe that
$$\overleftrightarrow{\bigcap_{n\in\w}\bar\uparrow\bar\downarrow h^{-1}(U_n)}=\overleftrightarrow{\bar\uparrow\bar\downarrow h^{-1}({\uparrow}V_y)}=\overleftrightarrow{h^{-1}({\uparrow}V_y)}
\subset\overline{h^{-1}({\uparrow}V_y)}\subset X\setminus h^{-1}(V_s)\subset X\setminus \{x\}.$$

Finally, assume that the condition (3) holds. In this case we can replace $V_y$ by a smaller neighborhood and additionally assume that $s\notin {\uparrow}\overline{{\uparrow}\overline{{\uparrow} V_y}}$.
By the continuity of the semilattice operation at $y$, there exists a decreasing sequence $(U_n)_{n\in\w}$ of open neighborhoods of $y$ such that $U_n\subset U_nU_n\subset U_{n-1}\cap {\uparrow} V_y$ for $n\in\w$.
Replacing each neighborhood $U_n$ by ${\uparrow}U_n$, we can assume that $U_n={\uparrow}U_n$ is an upper set. Then $W_n:=h^{-1}(U_n)$ is an open upper set in $X$ such that $W_nW_n\subset W_{n-1}\cap h^{-1}({\uparrow}V_y)$ for all $n\in\w$. Lemma~\ref{l:cccl} guarantees that $(\bar\downarrow W_n)\cdot(\bar\downarrow W_n)\subset\bar\downarrow (W_nW_n)\subset \bar\downarrow W_{n-1}$ and hence $({\uparrow}\bar\downarrow W_n)\cdot({\uparrow}\bar\downarrow W_n)\subset{\uparrow}\bar\downarrow W_{n-1}$, which implies that $F=\bigcap_{n\in\w}{\uparrow}\bar\downarrow W_n$ is a subsemilattice of $X$ with $F={\uparrow}F$. By Lemma~\ref{l:w}, the semilattice
$\bigcap_{n\in\w}\bar\uparrow\bar\downarrow W_n\subset \bigcap_{n\in\w}{\uparrow}\bar\downarrow W_n=F$ is not empty.
By our assumption, $\overleftrightarrow{F}\subset{\uparrow}\overline{{\uparrow}F}=
{\uparrow}\overline{F}$.  Then
$$\overleftrightarrow{\bigcap_{n\in\w}\bar\uparrow\bar\downarrow h^{-1}(U_n)}\subset \overleftrightarrow{\bigcap_{n\in\w}{\uparrow}\bar\downarrow W_n}\subset\overleftrightarrow{F}\subset{\uparrow}\overline{F}.$$It remains to show that $x\notin {\uparrow}\overline{F}$. Observe that $F\subset {\uparrow}\bar\downarrow h^{-1}({\uparrow}V_y)$.
We claim that $\bar\downarrow h^{-1}({\uparrow}V_y)\subset{\uparrow}\overline{h^{-1}({\uparrow}V_y)}$.
Indeed, for any $a\in \bar\downarrow h^{-1}({\uparrow}V_y)$, we can find a countable chain $C\subset h^{-1}({\uparrow}V_y)$ with $a=\inf C$ and observe that $C$ is a semilattice such that  ${\uparrow}C\subset {\uparrow}h^{-1}({\uparrow}V_y)=h^{-1}({\uparrow}V_y)$. By our assumption, $\overleftrightarrow{C}\subset{\uparrow}\overline{{\uparrow}C}$ and hence $$a=\inf C\in\overleftrightarrow{C}\subset{\uparrow}\overline{{\uparrow}C}\subset {\uparrow}\overline{h^{-1}({\uparrow}V_y)}.$$

Therefore $\bar \downarrow h^{-1}({\uparrow}V_y)\subset {\uparrow}\overline{h^{-1}({\uparrow}V_y)}$ and $F\subset{\uparrow}\bar{\downarrow}h^{-1}({\uparrow}V_y)\subset {\uparrow}{\uparrow}\overline{h^{-1}({\uparrow}V_y)}=
{\uparrow}\overline{h^{-1}({\uparrow}V_y)}$. By the continuity and monotonicity of $h$,
$$h(F)\subset h({\uparrow}\overline{h^{-1}({\uparrow} V_y)})\subset {\uparrow}h(\overline{h^{-1}({\uparrow} V_y)}\subset {\uparrow}\overline{h(h^{-1}({\uparrow} V_y))}\subset {\uparrow}\overline{{\uparrow}V_y}.$$ Then
$h(\overline{F})\subset\overline{h(F)}\subset \overline{{\uparrow}\overline{{\uparrow}V_y}}$ and
$h({\uparrow}\overline{F})\subset{\uparrow} \overline{{\uparrow}\overline{{\uparrow}V_y}}\subset Y\setminus\{s\}$, which implies the desired non-inclusion $x\notin{\uparrow}\overline{F}$.
\smallskip

Now it is legal to apply Lemma~\ref{l:super} to conclude that the set $S$ is closed in $Y$.
\end{proof}

\section{Proof of Theorem~\ref{main}}\label{s:proof}

In this section we shall prove a more general version of Theorem~\ref{main}.

Let $h:X\to Y$ be a continuous homomorphism from a $c$-complete topological semilattice $X$ to a Hausdorff topological semigroup $Y$. We shall prove that the image $S:=h(X)$ is closed in $Y$ if one of the following conditions is satisfied:
\begin{enumerate}
\item $X$ is $k$-complete;
\item $X$ is $s$-complete and $Y$ is well-separated;
\item $X$ or $Y$ is down-open;
\item $X$ or $Y$ is a topological lattice.
\end{enumerate}
\smallskip

1. Assume that $X$ is $k$-complete. Then for each closed subset $F\subset X$ and each non-empty chain $C\subset X$ we get $\{\inf C,\sup C\}\subset\bar C\subset \bar F=F$, which means that $F$ is chain-closed.
In particular, for any open subset $U\subset X$ the closed set $\overline{{\uparrow}U}$ is chain-closed and hence contains the chain-closure $\overleftrightarrow{U}$ of the set $U$. Now we see that the condition (2) of Theorem~\ref{t:main} is satisfied and hence $h(X)$ is closed in $Y$.
\smallskip

2. Assume that $X$ is $s$-complete and $Y$ is well-separated. Then for any non-empty subsemilattice $S\subset X$ we get $\inf S\in\overline{{\uparrow}S}$. By the $c$-completeness of $X$, the closed upper set ${\uparrow}\inf S$ is chain-closed and hence $\overleftrightarrow{S}\subset {\uparrow}\inf S={\uparrow}\overline{S}$. Now we can apply Theorem~\ref{t:main}(3) and conclude that $h(X)$ is closed in $Y$.
\smallskip

3. If $X$ or $Y$ is down-open, then $h(X)$ is closed in $Y$ by Theorem~\ref{t:main}(1)
\smallskip

4. If $X$ or $Y$ is a topological lattice, then $X$ or $Y$ is down-open according to Corollary~\ref{c:l=>do}. By the preceding item $h(X)$ is closed in $Y$.

\section{$\hC$-Closed topological semilattices}\label{s:diagram}

In this section we prove the implications drawn in the diagram at the end of the introduction. These implications can be derived from Corollary~\ref{c:l=>do} (saying that each topological lattice is down-open), Lemma~\ref{l:well} (establishing the embeddings of the classes $\TsLo\subset\TsLw\subset \TsL$), Theorem~\ref{t:BB} (implying that a Hausdorff topological semilattice is $k$-complete if and only if it has compact maximal chains) and the following two propositions.

\begin{proposition}\label{p:diag} For any Hausdorff topological semilattice $X$ the following statements hold.
\begin{enumerate}
\item If $X$ is $k$-complete, then $X$ is $s$-complete.
\item If $X$ is $s$-complete, then $X$ is $c$-complete.
\item If $X$ has $c$-complete maximal chains, then $X$ is $c$-complete.
\end{enumerate}
\end{proposition}

\begin{proof}
1. Suppose $X$ is $k$-complete. To show that $X$ is $s$-complete, fix any non-empty subsemilattice $S\subset X$ and observe that  $\overline{{\uparrow}S}$ is a closed subsemilattice in $X$. Using Zorn's Lemma, choose any maximal chain $M$ in $\overline{{\uparrow}S}$. By the $k$-completeness of $X$ the chain $M$ has $\inf M\in\overline{M}\subset \overline{{\uparrow} S}$. We claim that
$a:=\inf M$ is a lower bound for $S$. Given any element $s\in S$, observe that $as\le a$ and $as\in\overline{{\uparrow}S}$. By the maximality of the chain $M$, $as\in M$ and hence $a\le as$, which means that $a=as\le s$ and hence $a$ is a lower bound for the semilattice $S$. To see that $a$ is the largest lower bound for $S$, take any lower bound $b$ for $S$ and observe that $S\subset{\uparrow}b$ implies $a\in \overline{{\uparrow}S}\subset{\uparrow}b$ and hence $b\le a$. So, $\inf S=a\in \overline{{\uparrow}}S$.

On the other hand, the $k$-completeness of $X$ guarantees that each non-empty chain $C\subset X$ has $\sup C\in \overline{C}\subset \overline{{\downarrow} C}$. Hence $X$ is a $s$-complete semilattice.
\smallskip

2. Suppose that $X$ is $s$-complete semilattice. Let $F\subset X$ is an arbitrary closed upper set and $C$ be an arbitrary chain in $F$. Observe that the upper set ${\uparrow} C\subset F$ and its closure $\overline{{\uparrow}C}\subset \overline{F}=F$ are subsemilattices in  $X$. Since $X$ is $s$-complete, the semilattice ${\uparrow}C$ has $\inf {\uparrow}C\in \overline{{\uparrow}C}\subset F$. It is clear that $\inf C=\inf {\uparrow} C\in F$.

Since $X$ is $s$-complete semilattice each non-empty chain $C\subset X$ has $\sup C\in\overline{{\downarrow}C}$. Hence $X$ is $c$-complete semilattice.
\smallskip

3. Let $X$ be a semilattice with $c$-complete maximal chains. Let $F\subset X$ be an arbitrary closed upper set and $C$ be an arbitrary chain in $F$. Using Zorn's Lemma, extend the chain $C$ to a maximal chain $M\subset X$. Since the closure of a chain in a semilattice is a chain, the maximal chain $M$ is closed in $X$.

Consider the upper set $U:=M\cap {\uparrow}C\subset M\cap F$ in $M$ and its closure $\overline{U}\subset \overline{M\cap F}\subset \overline{M}\cap\overline{F}=M\cap F$. By Proposition~\ref{p:Ward}, the linear pospace $M$ is a topological lattice. Applying Lemma~\ref{l:updown}(3) (to the continuous semilattice operation $M\times M\to M$, $(x,y)\mapsto \sup\{x,y\}$), we conclude that $\overline{U}$ is an upper set in $M$.

The $c$-completeness of the maximal chain $M$ ensures that the chain $C$ has the greatest lower bound $\inf_M C\in \overline{U}\subset F$ in $M$. We claim that $\inf_M C$ is the greatest lower bound of $C$ in $X$. Given any lower bound $b\in X$ for $C$, we conclude that $U=M\cap {\uparrow}C\subset{\uparrow}b$ and hence $\inf_M C\in \overline{U}\subset{\uparrow}b$. Then $b\le \inf_M C$, which means that $\inf C=\inf_M C\in F$.

On the other hand, the $c$-completeness of $M$ ensures that the chain $C\subset M$ has the least upper bound $\sup_M C\in\overline{M\cap{\downarrow}C}$ in $M$. We claim that $\sup_M C$ is the least upper bound for $C$ in $X$.
Given any upper bound $b\in X$ for $C$, observe that $C\subset{\downarrow}b$ and hence $\sup_M C\in\overline{M\cap{\downarrow}C}\subset {\downarrow}b$, which means that $\sup_M C\le b$ and $\sup_M C$ is the least upper bound for $C$ in $X$. Then $\sup C=\sup_M C\in\overline{M\cap{\downarrow}C}\subset\overline{{\downarrow}C}$, which completes the proof of the $c$-completeness of the topological semilattice $X$.
\end{proof}

Proposition~\ref{p:diag} and the general version of the Theorem~\ref{main}, proved in Section~\ref{s:proof}, imply the following corollary.

\begin{corollary}\label{p:diag2} For any Hausdorff topological semilattice $X$ the following statements hold.
\begin{enumerate}
\item If $X$ is $k$-complete, then $X$ is ${\mathsf h}{:}\!\TsL$-closed.
\item If $X$ is $s$-complete, then $X$ is ${\mathsf h}{:}\!\TsLw$-closed.
\item If $X$ is $c$-complete, then $X$ is ${\mathsf h}{:}\!\TsLo$-closed.
\item If $X$ is $c$-complete and down-open, then $X$ is ${\mathsf h}{:}\!\TsL$-closed.
\end{enumerate}
\end{corollary}

\section{Some Comments and Open Problems}\label{s:last}

In this section we shall discuss some results and open problems related to Main Theorem~\ref{main}.

Theorems~\ref{t:Stepp}, \ref{t:BB}, \ref{t:GR}, \ref{main}  suggest the following intriguing open problems.

\begin{problem}\label{prob} Is any (regular) absolutely $H$-closed topological semilattice $X$ $c$-complete?
\end{problem}

\begin{problem} Is each $s$-complete (regular) Hausdorff topological semilattice $X$ absolutely $H$-closed?
\end{problem}

\begin{problem} Is a Hausdorff topological semilattice $X$ (absolutely) $H$-closed if all maximal chains in $X$ are $c$-complete?
\end{problem}

\begin{problem} Is a Hausdorff topological semilattice $X$ $s$-complete if all maximal chain in $X$ are $s$-complete?
\end{problem}

\begin{problem} Is each $c$-complete Hausdorff topological semilattice $s$-complete?
\end{problem}

\begin{problem} Is each absolutely $H$-closed topological semilattice chain-complete?
\end{problem}

The following example shows that $H$-closed topological semilattices need not be $c$-complete.

\begin{example}\label{Bard} There exists a topological semilattice $X$ such that
\begin{enumerate}
\item $X$ is metrizable, countable, and locally compact;
\item $X$ is $H$-closed;
\item $X$ is not chain-complete;
\item $X$ contains a closed upper set $U$, which is not chain-closed in $X$,
so $X$ is not $c$-complete;
\item there exists an injective continuous homomorphism $h:X\to Z$ to a compact Hausdorff topological semilattice $Z$ whose image $h(X)$ is not closed in $Z$.
\end{enumerate}
\end{example}

\begin{proof}
Let $\bar\IZ=\{-\infty\}\cup\IZ\cup\{+\infty\}$ be the set of integer numbers with attached elements $-\infty,+\infty$ such that $-\infty<z<+\infty$ for all $z\in\IZ$. We endow $\bar\IZ$ with the semilattice operation of minimum. Let $\{0,1\}$ be the discrete two-element semilattice with the semilattice operation of minimum.
In the product semilattice $\bar\IZ\times\{0,1\}$, consider the subsemilattice $$X:=(\bar\IZ\times\{0\})\cup(\IZ\times\{1\}).$$
Endow $X$ with the topology
$$
\begin{aligned}
\tau=\;&\{U\subset X:(-\infty,0)\in U\Ra(\exists n\in\IZ\;\forall m\le n\;\; (-2m,0)\in U)\}\;\cup\\
 &\cup\{U\subset X:(+\infty,0)\in U\Ra(\exists n\in\IZ\;\forall m\ge n\;\; (2m,0)\in U)\}.
\end{aligned}
$$

By \cite[Theorem~4]{Bardyla-Gutik-2012}, the topological semilattice $X$ is $H$-closed. On the other hand, the upper closed set $\IZ\times\{1\}$ is not chain-closed in $X$ as $(-\infty,0)=\inf (\IZ\times\{1\})\notin\IZ\times\{1\}$. So, $X$ is not $c$-complete.
Also the set $\IZ\times\{1\}$ has no upper bound in $X$, so the semilattice $X$ is not chain-complete.

Next, endow $\bar\IZ$ with the compact metrizable topology $$
\begin{aligned}
\tau_k=\;&\{U\subset X:(-\infty,0)\in U\Ra(\exists n\in\IZ\;\forall m\le n \;\;(-m,0)\in U)\}\;\cap\\
 &\cap\{U\subset X:(+\infty,0)\in U\Ra(\exists n\in\IZ\;\forall m\ge n\;\; (m,0)\in U)\}.
\end{aligned}
$$ Then the identity map $\id:X\rightarrow \bar\IZ\times\{0,1\}$ is an injective continuous homomorphism whose image $\id(X)$ is not closed in $\bar\IZ\times\{0,1\}$.
\end{proof}

In \cite{Yokoyama2013} Yokoyama asked the question: {\it Is each $H$-closed pospace chain-complete?} The following example gives a negative answer to this question.

\begin{example}
Let $\mathbb I:=[0,1]$ be the unit interval endowed with the usual topology $\tau$ and let $L=\{1/n:n\in\mathbb N\}$. Let $\tau_{1}$ be the topology on $\mathbb I$,  generated by the base $\{V\setminus L:V\in\tau\}$. On the space $X$ consider the partial order $\preceq$ in which $x\preceq y$ iff either $x=y$ or $x,y\in L$ and $x\le y$. By \cite[3.12.5]{Engelking1989}, $X=([0,1],\tau_{1})$ is an $H$-closed topological space, which implies that $(X,\preceq)$ is an absolutely $H$-closed pospace. On the other hand, $L$ is a maximal chain in $(X,\leq)$ without lower bound in $X$.
\end{example}

\end{document}